\documentclass[a4paper,11pt]{amsart}
\usepackage{amssymb, amsmath, amsfonts, amscd}
\usepackage{mathrsfs, latexsym}
\usepackage[all,ps,cmtip]{xy}
\usepackage{array, longtable}
\usepackage{bm, enumitem}
\usepackage{xcolor}
\usepackage{comment}

\usepackage{todonotes}

\title[]{Total Cartier index of a bounded family}
\date{\today}

\author{Jingjun Han}
\address{\rm Shanghai Center for Mathematical Sciences, Fudan University, Shanghai 200438, China}
\email{hanjingjun@fudan.edu.cn}

\author{Chen Jiang}
\address{\rm Shanghai Center for Mathematical Sciences \& School of Mathematical Sciences, Fudan University, Shanghai 200438, China}
\email{chenjiang@fudan.edu.cn}

\newcommand{\Rr}{\mathbb{R}}

\newcommand{\Supp}{\operatorname{Supp}}

\newcommand{\Qq}{\mathbb{Q}}

\newcommand{\Ii}{\Gamma}

\newtheorem{thm}{Theorem}[section]
\newtheorem{lem}[thm]{Lemma}
\newtheorem{cor}[thm]{Corollary}

\newtheorem{Prbm}[thm]{Problem}

\theoremstyle{definition}
\newtheorem{defn}[thm]{Definition}

\newtheorem{exmp}[thm]{Example}

\theoremstyle{remark}

\begin{document}
\begin{abstract} 
We prove that the total Cartier index of a bounded family of projective varieties of klt type is bounded.
\end{abstract}
\maketitle

\numberwithin{equation}{section}

 \section{Introduction}
 Throughout this paper, we work over an algebraically closed field of characteristic $0$, for instance, the complex number field $\mathbb{C}$. All the varieties are quasi-projective and normal unless stated otherwise.

 \begin{defn}
 Let $X$ be a variety.
     The {\it total Cartier index} of $X$ is defined to be the minimal positive integer $k$ such that for any Weil $\mathbb{Q}$-Cartier divisor $D$ on $X$, $kD$ is Cartier. It is defined to be $\infty$ if such $k$ does not exists. 
 \end{defn}

 Recall that a variety $X$ is said to be {\it of klt type} if there exists an effective $\mathbb{R}$-divisor $B$ such that  $(X, B)$ is klt.

The following is the main result of this paper. 
 
\begin{thm}\label{thm:boundedness of global index}
    Let $\mathcal{P}$ be a bounded family of projective varieties of klt type. Then there exists a positive integer $N$ depending only on $\mathcal{P}$ such that the total Cartier index of $X$ is bounded from above by $N$ for any $X\in \mathcal{P}$.
\end{thm}

Theorem \ref{thm:boundedness of global index} completely solves a folklore question on the boundedness of total Cartier index in a bounded family (see \cite[Question 3.31]{HLS19}). We refer the readers to \cite[Theorem 1.10]{GKP16}, \cite[Lemma 2.24]{Bir19}, \cite[Lemma 7.14]{CH20},  \cite[Lemma 3.16]{Bir22}, and \cite[Theorem 1.10]{HLQ23} for some partial results. Our approach to Theorem \ref{thm:boundedness of global index} is based on \cite[Theorem 1.10]{HLQ23} and the following key lemma.


\begin{lem}\label{lem: Q-factorial bounded}
    Let $\mathcal{P}$ be a bounded family of projective varieties of klt type. Then 
    there exists a positive integer $m$ and a bounded family  
    $\mathcal{P}_\mathbb{Q}$  depending only on $\mathcal{P}$, such that for any $X\in \mathcal{P}$, there exists a variety $Y\in \mathcal{P}_\mathbb{Q}$ with the following properties:
    \begin{enumerate}
        \item there is a small projective birational morphism $Y\to X$; 
        \item $mK_Y$ is Cartier;
        \item $Y$ is $\frac{1}{m}$-lc.
    \end{enumerate}
\end{lem}

In applications, by Lemma \ref{lem: Q-factorial bounded}, we are able to replace a bounded family of varieties of klt type  by a bounded family of $\Qq$-Gorenstein
klt varieties. 
Note that one can further require that $Y$ is $\Qq$-factorial in Lemma~\ref{lem: Q-factorial bounded} by boundedness of crepant models \cite[Theorem~1.2 and Page 4]{Bir22}, but we will not use this fact in this paper.
Such step appears in proofs of many birational geometry results, see for example, \cite[Lemma 2.24]{Bir19}, \cite[Proposition 10.2]{Bir23}.

\begin{cor}\label{cor: bounded family epsilon-lc}
Let $\Ii\subseteq [0,1]$ be a finite set and let $\mathcal{P}$ be a bounded family of projective varieties. Then there exists a positive real number $\epsilon$ depending only on $\Ii$ and $\mathcal{P}$ such that for any klt pair $(X,B)$ where $X\in \mathcal{P}$ and $B\in \Ii$, $(X,B)$
is $\epsilon$-lc.
\end{cor}

\noindent\textbf{Acknowledgement}. The authors would like to thank Caucher Birkar, Lingyao Xie, and Minzhe Zhu for helpful comments and fruitful discussions. The first author was supported by NSFC for Excellent Young Scientists (\#12322102). The second author was supported by NSFC for Innovative Research Groups (\#12121001). The authors were supported by the National Key Research and Development Program of China (\#2023YFA1010600, \#2020YFA0713200), and are members of LMNS, Fudan University.

\section{Minimal model prgram}

We adopt basic definitions as in \cite{KM98, BCHM10, HX13} about the minimal model program. 
\begin{defn}
    Let $X$ be a variety. A (small) {\it $\mathbb{Q}$-factorialization} of $X$ is a variety $Y$ with a small proper birational map $Y\to X$ such that $Y$ is $\mathbb{Q}$-factorial.
\end{defn}

\begin{defn}
 A collection $\mathcal{P}$ of projective varieties is
said to be \emph{bounded} 
if there exists a projective morphism 
$\mathcal{X}\rightarrow S$
between schemes of finite type such that
each $X\in \mathcal{P}$ is isomorphic  to $\mathcal{X}_s$ 
for some closed point $s\in S$, where $\mathcal{X}_s$ is the fiber of $\mathcal{X}$ over $s$.
\end{defn}


 \begin{lem}\label{lem: semiample model is small}
     Let $X$ be a variety of klt type. Let $f: Y\to X$ be a proper birational morphism and let $E_Y$ be the sum of all $f$-exceptional prime divisors on $Y$. 
     Suppose that $K_Y+E_Y$ is nef over $X$, then $E_Y=0$ and in particular, $f$ is small. 
 \end{lem}
\begin{proof}
Let $(X, B)$ be a klt pair.
    Write $G=K_Y+E_Y-f^*(K_X+B)$, then $G$ is nef over $X$ and $-f_*G=B$ is effective. Then by the negativity lemma, $-G$ is effective. Note that every prime divisor of $E_Y$ has positive coefficient in $G$ as $(X, B)$ is klt, so $E_Y=0$. 
\end{proof}

 \begin{lem}\label{lem: gmm=q-factoril}
     Let $X$ be a variety of klt type. Let $g: W\to X$ be a log resolution and let $E$ be the sum of all $g$-exceptional prime divisors on $W$. 
Then there exists a good minimal model of $(W,E)$ over $X$, which is a $\mathbb{Q}$-factorialization of $X$. 
 \end{lem}
\begin{proof}
Let $(X, B)$ be a klt pair. 
Write $K_W+B_W=g^{*}(K_X+B)$. Then $(W,\epsilon B_W+(1-\epsilon) E)$ is klt for some $0<\epsilon\ll 1$.
By \cite[Theorem 1.2]{BCHM10}, there is a good minimal model $Y$ of $K_W+\epsilon B_W+(1-\epsilon) E$ over $X$. 
Since $K_W+\epsilon B_W+(1-\epsilon) E\sim_{\Rr,X} (1-\epsilon)(K_W+E)$, $Y$ is also a good minimal model of $(W,E)$ over $X$. In particular, $Y$ is $\Qq$-factorial and  $K_Y+E_Y$ is nef over $X$ where $E_Y$ is the strict transform of $E$ on $Y$ whose support consists of all exceptional divisors on $Y$ over $X$. Then by Lemma~\ref{lem: semiample model is small}, $Y\to X$ is small.
 This implies that $Y$ is a $\mathbb{Q}$-factorialization of $X$.
\end{proof}

 \begin{lem}\label{lem: mmp=relative mmp}
     Let $X\to S$ be a projective morphism and let $g: W\to X$ be a log resolution. Let $A$ be a  Cartier divisor on $X$ ample over $S$. Let $B$ be an effective $\mathbb{R}$-divisor on $W$ such that $(W, B)$ is lc. 
Then for any $t>2\dim X$,
a $(K_W+B+tg^*A)$-MMP over $S$ is the same as a $(K_W+B)$-MMP over $X$. In particular, there exists a good minimal models of $(W, B+tg^*A)$ over $S$ if and only if there exists a good minimal model of $(W, B)$ over $X$.
 \end{lem}

\begin{proof}
Note that we may choose a general effective $\mathbb{R}$-divisor $H\sim_{\mathbb{R}, S} tg^*A$ such that $(W, B+H)$ is lc. Hence a $(K_W+B+tg^*A)$-MMP over $S$ makes sense. 

  It is clear that any $(K_W+B)$-MMP over $X$ is a $(K_W+B+tg^*A)$-MMP over $S$ as it is $(g^*A)$-trivial. 
    
Conversely, it suffices to show that 
any $(K_W+B+tg^*A)$-MMP over $S$  is $(g^*A)$-trivial.
Let $R$ be a $(K_W+B+H)$-negative extremal ray over $S$, then it is also a $(K_W+B)$-negative extremal ray as $H$ is nef over $S$. 
By the length of extremal rays \cite[Theorem~18.2]{Fuj11}, $R$ is generated by a rational curve $C$ such that $-(K_W+B)\cdot C\leq 2\dim X$. This implies that $H\cdot C< -(K_W+B)\cdot C\leq 2\dim X<t$.  So $C$ is contracted by $g$ as $H\sim_{\mathbb{R}, S} tg^*A$. In particular, every 
 $(K_W+B+H)$-negative extremal ray over $S$ is a 
$(K_W+B)$-negative extremal ray over $X$. Hence 
    a $(K_W+B+tg^*A)$-MMP over $S$ is the same as a relative $(K_W+B)$-MMP on $W$ over $X$.
    
The last sentence follows directly from \cite[Corollary~2.9]{HX13}.
\end{proof}

\section{Total Cartier index}

  \begin{lem}\label{lem: tci=q-factorial}
 Let $X$ be a projective variety of klt type and let $f: X'\to X$ be a small proper birational map. 
 Then the total Cartier index of $X$ is no greater than the total Cartier index of $X'$.
 \end{lem}

\begin{proof}
    Denote $N$ to be the total Cartier index of $X'$. Fix a Weil $\mathbb{Q}$-Cartier divisor $D$ on $X$. It suffices to show that $ND$ is Cartier. Denote $D'$ by the strict transform of $D$ on $X'$. Since $f: X'\to X$ is small, $D'=f^*D$. By assumption, $ND'$ is Cartier.
    
    By assumption, there exists a boundary $B$ such that $(X, B)$ is klt. So we may write $K_{X'}+B'=f^*(K_X+B)$ and $(X', B')$ is klt. 
Moreover, since $f$ is birational, there exists an effective Cartier divisor $G$ on $X'$ such that $-G$ is ample over $X$. 
Indeed, by \cite[Theorem~II.7.17]{GTM52}, $X'$ is the blowup of $X$ along a coherent sheaf of ideals $\mathcal{I}$, and the inverse image ideal is invertible which is of the form $\mathcal{O}_{X'}(-G)$ where $G$ is an effective Cartier divisor. Here $-G$ is ample over $X$ by \cite[Proposition~II.7.10]{GTM52}

Then there exists a suffciently small positive rational number $\epsilon$ such that 
$(X', B'+\epsilon G)$ is klt. Note that $-(K_{X'}+B'+\epsilon G)$ is ample over $X$, so $f: X'\to X$ is a contraction of a $(K_{X'}+B'+\epsilon G)$-negative extremal face as in the cone theorem \cite[Theorem~3.7]{KM98}. Recall that $D'=f^*D$ and $ND'$ is Cartier,  so $ND$ is Cartier by \cite[Theorem~3.7(4)]{KM98}. 
\end{proof}

\begin{proof}[Proof of Lemma \ref{lem: Q-factorial bounded}]
    As $\mathcal{P}$ is bounded, there is a quasi-projective scheme $\mathcal{X}$
and a projective morphism $\mathcal{X}\to S$, where $S$ is a disjoint union of finitely many varieties 
such that
for every $X\in \mathcal{P}$, there is a closed point $s \in S$  such that ${\mathcal{X}_{s}} \simeq X$. 
We may assume that the set of points $s$ corresponding to $X\in \mathcal{P}$ is dense in $S$. 
By the Noetherian induction, we may replace $S$ by a non-empty open subset, and we may assume that $S$
is a smooth affine variety.

Now consider a log resolution $g: \mathcal{W}\to \mathcal{X}$. Denote by $\mathcal{E}$ the sum of all $g$-exceptional prime divisors on $\mathcal{W}$. By the Noetherian induction, we may replace $S$ by a non-empty open subset and 
assume that for every $s\in S$, $\mathcal{W}_s$ is a log resolution of $\mathcal{X}_s$ and $\mathcal{E}_s$ is the sum of  all exceptional prime divisors on $\mathcal{W}_s$.

Fix a Cartier divisor $\mathcal{A}$ on $\mathcal{X}$ ample over $S$ and take a general effective $\mathbb{Q}$-divisor $\mathcal{H}\sim_{\mathbb{Q}}(2\dim \mathcal{X}+1)g^*\mathcal{A}$. We may assume that $(\mathcal{W}, \mathcal{E}+\mathcal{H})$ is log smooth over $S$. 

For a point $s\in S$ corresponding to $X\in \mathcal{P}$,  $(\mathcal{W}_s, \mathcal{E}_s)$ has a good minimal model over $\mathcal{X}_s$ by Lemma~\ref{lem: gmm=q-factoril}.
Then $(\mathcal{W}_s, \mathcal{E}_s+\mathcal{H}_s)$ has a good minimal model by Lemma~\ref{lem: mmp=relative mmp}. 
Since the set of such $s$ is dense in $S$, by applying  \cite[Theorem~1.2]{HMX18}, 
$(\mathcal{W}, \mathcal{E}+\mathcal{H})$ has a good minimal model over $S$. 
By \cite[Corollary~2.9]{HX13}, we may run a $(K_\mathcal{W}+\mathcal{E}+\mathcal{H})$-MMP  over $S$ with a scaling of an ample divisor, which terminates to a good minimal model $\mathcal{Y}$ over $S$. By the Noetherian induction, we may replace $S$ by a non-empty open subset and
 assume that for any $s\in S$,  $\mathcal{Y}_s$ is a semi-ample model of $({\mathcal{W}_s}, \mathcal{E}_s+\mathcal{H}_s)$. 
By Lemma~\ref{lem: mmp=relative mmp} again, this MMP is over $\mathcal{X}$.
Therefore, for a closed point $s\in S$ corresponding to $X\in \mathcal{P}$, 
$\mathcal{Y}_{s}$ is a semi-ample model of $(\mathcal{W}_{s}, \mathcal{E}_{s})$ over $X$. 
Then by Lemma~\ref{lem: semiample model is small}, 
 $\mathcal{Y}_{s}\to X$ is small. 
Here we remark that $\mathcal{Y}$ is $\Qq$-factorial by the construction, but  $\mathcal{Y}_{s}$ is not necessarily $\Qq$-factorial in general. 

We can just take the set of such $\mathcal{Y}_{s}$ to be $\mathcal{P}_\mathbb{Q}$. 
By construction, there exists a positive integer $m$ depending only on $\mathcal{P}$ such that $mK_{\mathcal{Y}}$ is Cartier. So $mK_{\mathcal{Y}_s}$ is Cartier.
Since $X$ is of klt type, $\mathcal{Y}_s$ is klt, and in particular, $\mathcal{Y}_s$ is $\frac{1}{m}$-lc. 
\end{proof}

\begin{proof}[Proof of Theorem \ref{thm:boundedness of global index}]
By Lemma \ref{lem: Q-factorial bounded} and Lemma \ref{lem: tci=q-factorial}, possibly replacing $\mathcal{P}$, we may assume that each $X$ is $\epsilon$-lc for some fixed $\epsilon>0$. The Theorem follows from \cite[Theorem 1.10]{HLQ23}.
\end{proof}

\begin{proof}[Proof of Corollary \ref{cor: bounded family epsilon-lc}]
By \cite[Lemma 5.3]{HLS19} and \cite[Theorem 5.6]{HLS19}, there exist positive real numbers $a_1,a_2,\ldots,a_m$ with $\sum_{i=1}^m a_i=1$, and a finite set $\Ii_0\subseteq [0,1]\cap \Qq$ satisfying following.  

For each klt pair $(X,B)$, where $X\in \mathcal{P}$, and $B\in \Ii$, there exists $\Qq$-divisors $B_i\in \Ii_0$ on $X$, such that 

\begin{itemize}
\item $B=\sum_{i=1}^m a_iB_i$, and
    \item each $(X,B_i)$ is klt.
\end{itemize}
By Theorem \ref{thm:boundedness of global index}, there exists a positive integer $N$ depending only on $\mathcal{P}$, such that $N(K_X+B_i)$ is Cartier for each $i$. In particular, $(X,B_i)$ is $\frac{1}{N}$-lc for each $i$. Thus $(X,B)$ is $\frac{1}{N}$-lc. We may let $\epsilon=\frac{1}{N}$ as required.
\end{proof}

\section{Examples and Open problems}
We collect several examples to show previous known results as well as to discuss whether the conditions in Theorem \ref{thm:boundedness of global index} are necessary. These examples should be well-known to experts.


\begin{exmp}\label{ex1}\begin{enumerate}
  \item If $X$ is a quotient of a smooth variety by a finite group $G$, then the total Cartier index of $X$ is no greater than $|G|$. 

  \item If $X$ is an isolated hypersurface singularity of dimension at least $3$, then the total Cartier index of $X$ is $1$ by \cite[Theorem~5.2]{Milnor}. 

\item If $X$ is a terminal singularity of dimension $3$, then the total Cartier index of $X$ is the Cartier index of $K_X$ by \cite[Lemma~5.2]{Kaw88}.
\end{enumerate}
\end{exmp}

\begin{exmp}
\begin{enumerate}
\item If $X$ is of klt type, then the total Cartier index of $X$ is finite by \cite[Theorem 1.10]{GKP16} or \cite[Lemma 7.14]{CH20}.
\item More generally, if $X$ has rational singularities, then by the proof of \cite[Lemma 1.12]{Kaw88}, the total Cartier index of $X$ is finite.
\end{enumerate}
\end{exmp}

\begin{exmp}
\begin{enumerate}
\item Suppose that $(X,\Supp B)$ is log bounded, $(X,B)$ is klt and the coefficients of $B$ belong to a finite set of rational numbers, then the Cartier index of $K_X+B$ is bounded from above by \cite[Lemma 3.16]{Bir22}.

\item The set of $\epsilon$-lc Fano varieties of dimension $d$ are bounded by \cite{Bir21} for any fixed positive integer $d$ and $\epsilon>0$. By \cite[Theorem 1.10]{HLQ23}, the total Cartier index of such varieties are bounded from above.

\item The total Cartier indices of terminal Fano threefolds are bounded from above by $840$ by \cite[Proposition~2.4]{CJ16} and \cite[Lemma~5.2]{Kaw88}.


\end{enumerate}
\end{exmp}

  \begin{exmp}
  \begin{enumerate}
\item Let $X$ be a cone over an elliptic curve, then by \cite[Proposition 3.14]{Kol13}, $X$ is lc and not of klt type, and the total Cartier index of $X$ is $\infty$. Thus we could not relax the assumption ``of klt type'' in Theorem \ref{thm:boundedness of global index} to ``lc''.

\item  The set of all smooth varieties is not bounded. However, the total Cartier index of any smooth variety is $1$. Thus, ``$\mathcal{P}$ be a bounded family'' is not an necessary condition in Theorem \ref{thm:boundedness of global index}.
  \end{enumerate}
  \end{exmp}

We propose some open problems related to bounded families of varieties. 

\begin{Prbm}
Does Theorem \ref{thm:boundedness of global index} still hold if we relax the assumption ``of klt type'' to ``with rational singularities''?
\end{Prbm}

\begin{Prbm}
Is being of klt type an open condition in a flat family of varieties?
\end{Prbm}

We remark that being klt is not an open condition because being $\mathbb{Q}$-Gorenstein is not open by \cite[Examples 9.1.7, 9.1.8]{Ishii18}; but it is an open condition in a flat family of $\mathbb{Q}$-Gorenstein varieties by \cite[Theorem~A]{ST23}.

 Recall that a variety $X$ is said to be {\it of $\epsilon$-lc type} if there exists an effective $\mathbb{R}$-divisor $B$ such that  $(X, B)$ is $\epsilon$-lc. A projective variety $X$ is said to be {\it of $\epsilon$-Fano type} if there exists an effective $\mathbb{R}$-divisor $B$ such that  $(X, B)$ is $\epsilon$-lc and $-(K_X+B)$ is ample.

\begin{Prbm}
Let $\mathcal{P}$ be a bounded family of projective varieties of klt type. Does there exist a positive real number $\epsilon$ such that for any $X\in \mathcal{P}$, $X$ is of $\epsilon$-lc type?
\end{Prbm}

\begin{Prbm}
Let $\mathcal{P}$ be a bounded family of projective varieties of Fano type. Does there exist a positive real number $\epsilon$ such that for any $X\in \mathcal{P}$, $X$ is of $\epsilon$-Fano type?
\end{Prbm}

\begin{Prbm}[{cf. \cite[Definition~1.6, Remark~1.8]{JW24}}]
Let $\mathcal{P}$ be a bounded family of projective klt Calabi--Yau varieties. Does there exist a positive integer $N$ such that for any projective klt Calabi--Yau variety $Y$ which is birational to some $X\in \mathcal{P}$, the total Cartier index of $Y$ is bounded by $N$?  Or more generally, does there exist a positive integer $N$ such that for any projective klt Calabi--Yau variety $Y$ of a fixed dimension $d$, the total Cartier index of $Y$ is bounded by $N$?  
\end{Prbm}

\end{document}